\documentclass[10pt,reqno]{amsart}
\usepackage[hmarginratio=1:1,vmarginratio={1:1},margin=3.3cm]{geometry}
\usepackage[british]{babel}
\usepackage[utf8]{inputenc}

\usepackage{graphicx,adjustbox}

\usepackage[pagebackref,hyperfootnotes]{hyperref}

\usepackage{color,units}

\usepackage[dvipsnames]{xcolor}
\hypersetup{
	colorlinks,
	linkcolor={red!50!black},
	citecolor={green!50!black},
	urlcolor={blue!80!black}
}

\usepackage{amsmath,amsfonts,amssymb,amsthm,relsize,mathrsfs,tikz-cd,mathtools,tensor,mathdots,slashed,epsfig,xypic,thmtools,changepage}
\usepackage[capitalize,nameinlink,noabbrev]{cleveref}
\usepackage[titletoc]{appendix}

\usepackage[textsize=small,prependcaption]{todonotes}
\xyoption{curve}

\usepackage{enumitem}
\setlist[enumerate]{label=(\roman*),leftmargin=0.8cm}

\usepackage{emptypage}

\setuptodonotes{color=green!40,caption={TODO}}
\reversemarginpar

\newcommand{\CC}{\mathbb{C}}

\newcommand{\PP}{\mathbb{P}}

\newcommand{\calH}{\mathcal{H}}
\newcommand{\calL}{\mathcal{L}}

\newcommand{\calV}{\mathcal{V}}

\newcommand{\mfh}{\mathfrak{h}}
\newcommand{\mfg}{\mathfrak{g}}

\DeclarePairedDelimiter{\ip}{\langle}{\rangle}

\newcommand{\acts}{\curvearrowright}

\newcommand{\Lap}{\Delta}

\newcommand{\rk}{\operatorname{rk}} 
\newcommand{\vol}{\operatorname{vol}} 

\newcommand{\Lie}{\operatorname{Lie}}

\newcommand{\g}{\mathfrak{g}} 
\newcommand{\h}{\mathfrak{h}} 
\renewcommand{\k}{\mathfrak{k}} 
\newcommand{\ad}{\operatorname{ad}} 

\newcommand{\Ham}{\operatorname{Ham}} 
\newcommand{\Ric}{\operatorname{Ric}} 

\newcommand{\GL}{\operatorname{GL}} 
\newcommand{\SL}{\operatorname{SL}} 
\newcommand{\U}{\operatorname{U}} 

\newcommand{\Hom}{\operatorname{Hom}} 
\newcommand{\End}{\operatorname{End}} 
\newcommand{\Aut}{\operatorname{Aut}} 
\newcommand{\Isom}{\operatorname{Isom}} 
\newcommand{\Symp}{\operatorname{Symp}} 

\newcommand{\fr}{\mathcal{F}} 

\newcommand{\contr}{\Lambda}

\newcommand{\del}{\partial}
\newcommand{\delbar}{\overline{\partial}}
\newcommand{\deldelbar}{\partial \overline{\partial}}
 
\newcommand{\rest}[2]{\left.#1\right|_{#2}}

\newcommand{\id}{{\boldsymbol{1}}} 

\newtheorem{theorem}{Theorem}[section]
\newtheorem{lemma}[theorem]{Lemma}
\newtheorem{corollary}[theorem]{Corollary}
\newtheorem{proposition}[theorem]{Proposition}
\newtheorem{conjecture}[theorem]{Conjecture}

\theoremstyle{definition}
\newtheorem{definition}[theorem]{Definition}
\newtheorem{example}[theorem]{Example}

\newtheorem{remark}[theorem]{Remark}

\newtagform{Eq}{(Eq. }{)}
\usetagform{Eq}

\makeatletter
\renewcommand*{\eqref}[1]{%
	\hyperref[{#1}]{\textup{\tagform@{\ref*{#1}}}}%
}
\makeatother

\let\oldchi\chi
\renewcommand{\chi}{\protect\raisebox{\depth}{$\oldchi$}}

\title{Canonical metrics on holomorphic fibre bundles}
\author{John Benjamin McCarthy}
	
\let\oldtocsection=\tocsection
\let\oldtocsubsection=\tocsubsection

\let\oldtocsubsubsection=\tocsubsubsection
\renewcommand{\tocsection}[2]{\hspace{0em}\oldtocsection{#1}{#2}}
\renewcommand{\tocsubsection}[2]{\hspace{2em}\oldtocsubsection{#1}{#2}}
\renewcommand{\tocsubsubsection}[2]{\hspace{4em}\oldtocsubsubsection{#1}{#2}}

\begin{document}

\maketitle

\begin{abstract}
	In this article we completely describe the existence of canonical metrics, known as optimal symplectic connections, on isotrivial K\"ahler fibrations. In this setting an optimal symplectic connection is induced from a Hermite--Einstein connection on the holomorphic principal bundle of relative automorphisms, and the Hitchin--Kobayashi correspondence asserts the existence of such a connection precisely when the principal bundle is polystable. Combined with results of Dervan and Sektnan this generates many new examples of cscK metrics on the total space of holomorphic fibre bundles. Our results indicate that in general the optimal symplectic connection equation should be viewed as a generalisation of the Hermite--Einstein equation to holomorphic fibrations where the complex structure of the fibres varies.
\end{abstract}

\setcounter{tocdepth}{3}
\vspace*{2em}
\section{Introduction}

An optimal symplectic connection, introduced by Dervan and Sektnan, is a special  metric on a holomorphic fibration which interpolates between special connections on bundles (Hermite--Einstein connections) and special metrics on varieties  (constant scalar curvature K\"ahler (cscK) metrics) \cite{dervan2019optimal,dervan2019moduli}. Indeed, on the one hand if a holomorphic fibration is the projectivisation of a holomorphic vector bundle, an optimal symplectic connection corresponds precisely to a Hermite--Einstein connection on the corresponding vector bundle. On the other hand, an optimal symplectic connection on a holomorphic fibration over a point is a constant scalar curvature K\"ahler metric on the unique fibre. 

For general holomorphic fibrations, it was proven by Dervan and Sektnan that when the fibration has cscK fibres and an appropriately twisted cscK metric on the base, the existence of an optimal symplectic connection implies the existence of a cscK metric in adiabatic K\"ahler classes of the total space provided the base and total space have discrete holomorphic automorphism groups. This is a generalisation for holomorphic fibrations of earlier results of Hong about the existence of cscK metrics on the total space of projectivisations of stable holomorphic vector bundles \cite{hong}, and of Fine in the case of fibrations where the fibres have discrete automorphism groups \cite{fine}.

The optimal symplectic connection equation,
$$p(\Lap_\calV \contr_{\omega_B} \mu^* F_\calH + \contr_{\omega_B} \rho_\calH) = 0,$$
is a differential equation for a relatively K\"ahler metric $\omega_X\in c_1(H)$ (which is cscK on each fibre), otherwise known as a \emph{symplectic connection}, on a holomorphic fibration $\pi: (X,H)\to (B,L)$ where $H$ is a relatively ample line bundle with respect to $\pi$.  More generally, $\omega_X$ is said to be an \emph{extremal symplectic connection} if $$p(\Lap_\calV \contr_{\omega_B} \mu^* F_\calH + \contr_{\omega_B} \rho_\calH)\in \mfh_V,$$ where $\mfh_V$ denotes the space of holomorphy potentials of vertical holomorphic vector fields vanishing somewhere (i.e. the Lie algebra of $\Aut_0(\pi) = \{g\in\Aut_0(X,L): \pi \circ g = \pi\}$). The optimal symplectic connection equation appears in the sub-leading order term in the adiabatic expansion of the scalar curvature of the K\"ahler form $\omega_X + k \omega_B$ on the total space of $X$ for $k\gg 0$, where it serves as the critical obstruction to perturbing this K\"ahler form to a cscK metric on $X$, while similarly the extremal symplectic connection appears in the corresponding construction of extremal K\"ahler metrics in $c_1\left(\pi^*L^{\otimes k}\otimes H\right)$ for $k\gg0$. We note that the relatively K\"ahler metric $\omega_X$ \emph{induces} a symplectic connection through giving a splitting of the tangent bundle of $X$, so much as in the vector bundle setting, one can dually consider metric-type structures of connection-type structures.

In this article, we investigate the optimal symplectic connection equation in the case of \emph{isotrivial} K\"ahler fibrations, where the complex structure is fixed as the fibres vary. This includes the previously understood cases of projectivisations of holomorphic vector bundles and of coadjoint orbits, but also includes any associated holomorphic fibre bundles to holomorphic principal bundles $P\to (B,L)$. In more detail, as we explain in Section \ref{sec3}, to a holomorphic fibre bundle $(X,H)\to (B,L)$ with model fibre $(Y,H_Y)$ admitting a cscK metric, letting $G=\Aut_0(Y,H_Y)$ (which is reductive as $c_1(H_Y)$ is assumed to admit a cscK metric), we obtain a principal $G$-bundle $P\to (B,L)$; conversely, such a principal $G$-bundle induces an associated bundle $(X,H)\to (B,L)$ with model fibre $(Y,H_Y)$.  Denoting $\Aut_0(\pi) = \{g\in\Aut_0(X,H): \pi \circ g = g\}$), our main result is the following.

\begin{theorem}\label{theorem:maintheorem}
Suppose $\Aut_0(\pi)$ is trivial.	If an isotrivial K\"ahler fibration $(X,H)\to (B,L)$ with cscK fibres arises as the associated bundle to a holomorphic principal $G$-bundle $P\to (B,L)$ with reductive structure group $G$, then $X$ admits an optimal symplectic connection if and only if $P$ admits a Hermite--Einstein connection.
\end{theorem}

The assumption that $\Aut_0(\pi)$ be trivial is equivalent to asking that $P$ be simple and that $\Lie G$ have trivial centre. In the presence of automorphisms, the two relevant equations are inequivalent in general, though the extremal symplectic connection condition agrees with the  Yang--Mills equation when the fibre $Y$ is a Fano manifold with $H_Y=-K_Y$ taken to be the anticanonical class (see Remark \ref{fano}). Using our main result, many new examples of cscK metrics on the total space of fibrations can be generated using the existence result of Dervan and Sektnan. This gives one of the most general methods of generating new examples of cscK metrics, which is typically a challenging problem.

\begin{corollary}\label{corollary:cscktotalspace}
	Let $P\to (B,L)$ be a holomorphic principal bundle with reductive fibre $G=K^\CC$ which admits a Hermite--Einstein structure, and suppose $G$ acts holomorphically on a polarised variety $(Y,H_Y)$ admitting a cscK metric, for which the $K$ action is isometric. Suppose futhermore that $B$ admits a cscK metric, and $B$ and the total space of the isotrivial fibration $(X,H)=P\times_G (Y,H_Y)$ have discrete automorphism groups. Then $X$ admits a cscK metric in the class $H+kL$ for $k\gg 0$. 
\end{corollary}

We note that when the complex structure of the fibres varies, one must require a twisting of the cscK metric on the base, but in our setting of isotrivial fibrations the correct twisting is trivial, so we only require $B$ to have a regular cscK metric. A concrete example is given in \cref{example-application} involving an $\SL(2,\mathbb C)$-bundle $P$.

The optimal symplectic connection condition also has a natural algebraic analogue, where it conjecturally corresponds to the notion of stability of a holomorphic fibration. Indeed a Hitchin--Kobayashi correspondence-type conjecture was proposed by Dervan and Sektnan:

\begin{conjecture}\label{conjecture:YTDfibrations}
	A holomorphic fibration $\pi : (X,H)\to (B,L)$ admits an optimal symplectic connection in $c_1(H)$ if and only if it is a polystable fibration.
\end{conjecture}

Some immediate progress was made towards this conjecture by Dervan and Sektnan, who showed that semistability of the fibration is a necessary condition for the existence of an optimal symplectic connection \cite{dervan2019moduli}. This result was strengthened by Hallam using geodesic analysis in the space of relatively K\"ahler metrics, who showed polystability with respect to certain fibration degenerations \cite{hallam2020geodesics}.

In the spirit of this central conjecture, we note that a Hitchin--Kobayashi correspondence has been proven for principal bundles \cite{subramanian1988einstein,anchouche2001einstein}. Thus as a corollary of our main result, we can link an algebraic stability condition to the existence of optimal symplectic connections on isotrivial fibrations.

\begin{corollary}
Let $\pi : (X,H)\to (B,L)$ be an isotrivial  fibration arising as the associated bundle of a holomorphic principal bundle, such that $\Aut_0(\pi)$ is trivial.  Then  $\pi : (X,H)\to (B,L)$ admits an optimal symplectic connection if and only if the principal bundle is polystable.
\end{corollary}

We note that in the context of \cref{corollary:cscktotalspace}, one can replace the assumption of the existence of a Hermite--Einstein structure with the polystable of the principal bundle. Many examples of such principal bundles can be generated by taking the frame bundle of a polystable vector bundle, so \cref{corollary:cscktotalspace} provides many new examples of cscK metrics in practice.

Finally we prove that a fibred product of fibrations (even in the non-isotrivial case) admits an optimal symplectic connection whenever the two factors do.

\begin{theorem}
	If $(X,\omega_X), (X',\omega_{X'})\to (B,\omega_B)$ are two fibrations with optimal symplectic connections, then the direct product metric $\omega_X + \omega_{X'}$ is optimal symplectic on the fibred product $Z=X\times_B X'\to B$. 
\end{theorem}

Combined with the existence results of optimal symplectic connections on isotrivial fibrations, this allows one to further generate new examples of optimal symplectic connections. We relate this to the more familiar story of the existence of Hermite--Einstein connections on direct sums of polystable holomorphic vector bundles. In contrast with that setting, where the slopes of the polystable vector bundles must be equal, there is no topological matching condition required between the fibrations in order for an optimal symplectic connection to exist.

\subsection{Outlook}

To complete the understanding of optimal symplectic connections and stability of fibrations in the isotrivial case, it would suffice to describe precisely the relation between stability of principal bundles and stability of associated fibrations. Schematically represented (where we assume $\Aut_0(\pi)$ is discrete throughout):

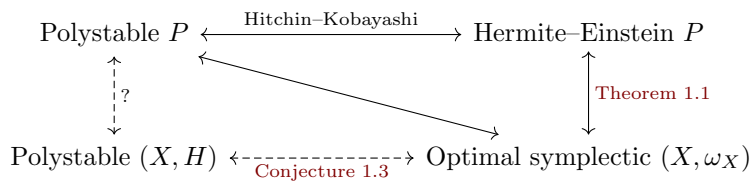
\begin{figure}[h]
	\centering
	\begin{tikzcd}[%
	,row sep = 3em
	,column sep=7em
	]
		\text{Polystable } P \arrow[leftrightarrow]{r}{\text{Hitchin--Kobayashi}} \arrow[leftrightarrow,dashed]{d}{?} \arrow[leftrightarrow]{dr} & \text{Hermite--Einstein } P \arrow[leftrightarrow]{d}{\text{\cref{theorem:maintheorem}}}\\
		\text{Polystable } (X,H) \arrow[leftrightarrow,dashed,swap]{r}{\text{\cref{conjecture:YTDfibrations}}} & \text{Optimal symplectic } (X,\omega_X)
	\end{tikzcd}
\end{figure}

The algebraic analogue of our main theorem, the left-hand vertical arrow, has been investigated previously in the case of projective bundles by Ross and Thomas in a the context of slope K-stability \cite{ross2006obstruction}. A reinterpretation of Ross and Thomas' result in the new framework of stability of fibrations shows that for such projective bundles, the algebraic analogue of our main theorem is true. We will return to this problem for general isotrivial fibrations in a sequel article.

Finally we comment on the nature of the optimal symplectic connection equation for non-isotrivial fibrations. This setting is now genuinely outside the realm of Hermite--Einstein connections, and our main result demonstrates that, at least in the case of smooth fibrations, the optimal symplectic connection equation should be thought of as a generalisation of the Hermite--Einstein equation to the setting of fibrations where the complex structure of the fibre varies. In this non-isotrivial setting it is important to modify the equation in order to allow strictly K-semistable fibres, as was observed by Ortu \cite{ortu2022optimal}. It is expected that our main result should find applications to the study of optimal symplectic connections on non-isotrivial deformations of isotrivial fibrations, as can appear in the work of Ortu.

It was observed by Dervan--Sektnan \cite{dervan2021uniqueness} that in the non-isotrivial but smooth Fano case, the optimal symplectic connection equation still admits a simplification to a Hermite--Einstein-type equation. Recent work on the algebraic side by Hattori \cite{hattori2022fibration} indicates that stable singular fibrations have mild singularities which may make the study of optimal symplectic connections with singular fibres tractable, and in the case where the smooth locus is isotrivial, one may hope to relate such singular fibrations to singular principal bundles, and singular Hermite--Einstein connections.

\vspace{4mm} \noindent {\bf Acknowledgements:}  The author wishes to thank their PhD supervisors Simon Donaldson and Ruadha\'i Dervan for useful comments and suggestions, and Michael Hallam and Annamaria Ortu for useful discussions. The author also thanks the referee for many helpful comments and corrections; we note that the first version of this paper did not assume $\Aut_0(\pi)$ be trivial in the statement of Theorem \ref{theorem:maintheorem} and the other main results, and the related errors in the proofs were discovered by the referee. This has led to corrections to some of the statements and proofs in the text, for which the author also thanks Ruadha\'i Dervan for helping correct. The author was funded by the EPSRC and the London School of Geometry and Number Theory.

\section{Preliminaries}

\subsection{Optimal symplectic connections}\label{osc-def}

In this section we recall the definition of an optimal symplectic connection as introduced in \cite{dervan2019optimal}. This is a notion of special metric or special connection on a holomorphic fibration which intertwines the bundle theory (Hermite--Einstein connections) and the K\"ahler theory (constant scalar curvature K\"ahler metrics). In particular the existence of an optimal symplectic connection is intimately related to the existence of cscK metrics on the total space of the fibration for adiabatic K\"ahler classes. 

\begin{definition}
	A \emph{smooth K\"ahler fibration} is a holomorphic surjective submersion $$\pi: (X,\omega_X) \to (B,\omega_B)$$ of complex manifolds where $\omega_B$ is a K\"ahler form on $B$ and $\omega_X$ is closed $(1,1)$-form on $X$ such that the restriction $\rest{\omega_X}{b}$ is a K\"ahler form on $X_b$ for every $b\in B$. 
\end{definition}

We will always denote the dimension of the base as $n$, and of the fibres as $m$. When the fibres are compact a smooth K\"ahler fibration is automatically a smooth fibre bundle by Ehresmann's lemma. Furthermore, since the two-form $\omega_X$ is non-degenerate in the vertical directions, there exists a horizontal subbundle $\calH\subset TX$ defined as the symplectic orthogonal complement of the vertical tangent bundle $\calV \subset TX$. Thus any smooth K\"ahler fibration comes with an Ehresmann connection $\calH$. We will refer to both $\calH$ and $\omega_X$ as a \emph{symplectic connection}. For more details of this construction see \cite[\S 1]{guillemin1996symplectic}\cite[\S 6]{mcduff2017introduction}. 

\begin{definition}
	The \emph{symplectic curvature} of $\omega_X$ is the curvature form $F_\calH \in \Omega^2(X, \Symp(\calV))$ of the Ehresmann connection $\calH$, which takes values in the symplectic vector fields of the fibres of $X$, and which is a basic form on $X$.
\end{definition}

In particular $F_\calH (v_1,v_2)\in \Symp(X_b)$ for horizontal lifts of every pair of vectors $v_1,v_2\in T_b B$ on $B$, where $\Symp(X_b)$ consists of the symplectic vector fields on $X_b$ with respect to $\rest{\omega_X}{b}$. The symplectic curvature is related to the original two-form $\omega_X$ in the following way, which is known as \emph{minimal coupling} in the symplectic fibrations literature.

\begin{theorem}[{\cite[\S 1.3]{guillemin1996symplectic}, \cite[Lem. 3.2]{dervan2019optimal}}]\label{thm:minimalcoupling}
	The symplectic curvature takes values in Hamiltonian vector fields. Furthermore, if $\mu^*: \Ham(\calV) \to C_0^{\infty}(X)$ denotes the map taking a vertical Hamiltonian vector field to its associated relative (mean zero) Hamiltonian function on $X$, and we abuse notation by identifying $\mu^* F_\calH$ with its pullback to $X$, which is a two-form which takes values in fibrewise Hamiltonian functions, then 
	$$\mu^* F_\calH = (\omega_X)_\calH + \pi^* \beta$$
	where $\beta\in \Omega^2(B)$ is some two-form on the base, and $(\omega_X)_\calH$ is the horizontal component of $\omega_X$ with respect to the orthogonal splitting it defines. 
\end{theorem}

We now recall the optimal symplectic connection equation. In this case we  assume $(X,\omega_X)\to (B,\omega_B)$ is a relatively cscK fibration, so that the scalar curvature $S(\rest{\omega_X}{b})$ is constant for every $b$, and that the fibres of $X$ have positive-dimensional automorphism group, as the optimal symplectic connection equation is vacuous when the automorphism group of the fibre is discrete.

To specify the optimal symplectic connection equation, we require several ingredients, which are explained in more detail in \cite[\S 3]{dervan2019optimal}. Firstly we note that associated to the relatively K\"ahler form $\omega_X$ there exists a relative Ricci form $\rho$ on $X$, defined by the expression
$$\rho = -i \deldelbar \log (\omega_X)_\calV^m$$
where $m$ is the dimension of the fibres of $X$. This is the curvature of the induced Hermitian metric on the determinant of the relative tangent bundle of $X\to B$. The vertical component of $\rho$ on a fibre $X_b$ is just the Ricci form of the K\"ahler metric $\rest{\omega_X}{b}$, but $\rho$ may have non-trivial horizontal component arising from the fact that $\deldelbar$ is being taken on the total space of $X$.

Given a horizontal two-form $\beta$ on $X$, such as $\rho_\calH$ or $\mu^* F_\calH$, one may contract with $\omega_B$. If this form is the pullback of a form on $B$, then the contraction $\contr_{\omega_B} \beta$ is just the pullback of the associated contraction on $B$. In general one defines
$$\contr_{\omega_B} \beta = n \frac{\beta_\calH \wedge \omega_B^{n-1}}{\omega_B^n}$$
where the quotient is taken in the line bundle $\det \calH^*$. 

On functions we have a vertical Laplace operator defined by
$$\Lap_\calV f = \contr_{\omega_X,\calV} (i\deldelbar f),$$ where for a $(1,1)$-form $\alpha$ on $X$ we define  the function $\contr_{\omega_X,\calV}(\alpha)$ by $$\contr_{\omega_X,\calV}\alpha|_{X_b} = \contr_{\omega_{X}|_{X_b}}(\alpha|_{X_b}).$$

Finally we recall the existence of an orthogonal projection operator $p: C^\infty(X) \to C_E^\infty(X)$, where the latter space refers to the space of relatively mean zero holomorphy potentials of $X$ with respect to the fibrewise K\"ahler metrics $\rest{\omega_X}{b}$. The holomorphy potentials on a K\"ahler manifold $M$ are defined by
$$\h = \ker \delbar \nabla^{1,0} : C^\infty (M,\CC) \to \Omega^{0,1}(T^{1,0}M).$$
Any such potential $f$ defines a holomorphic vector field $\xi$ on $M$ by $\xi^j = g^{j\bar k} \del_{\bar k} f$. The projection operator $p$ sends a function on the total space of $X$ to its $L^2$-projection onto the the space $C_E^\infty(X)$ of relative mean zero holomorphy potentials for each fibre $(X_b, \rest{\omega_X}{b})$ where the $L^2$ inner product on functions is defined on each fibre with respect to $\rest{\omega_X}{b}^m$. As part of the assumptions involved in the definition of an optimal symplectic connection, it is necessary for these spaces of holomorphy potentials to form a vector bundle $E$ over the base (which is automatic in the isotrivial case). This bundle was identified by Hallam with the bundle of relatively cscK metrics \cite{hallam2020geodesics}, and they observed that these spaces form a bundle precisely when the dimension of the space of holomorphy potentials is constant over $B$. 

With these ingredients set, we can now define an optimal symplectic connection:

\begin{definition}
	One says that a relatively cscK metric $\omega_X$ on a holomorphic fibration $X\to (B,\omega_B)$ is an \emph{optimal symplectic connection} if 
	$$p(\Lap_\calV \Lambda_{\omega_B} \mu^* F_{\calH} + \Lambda_{\omega_B} \rho_{\calH}) = 0,$$ and that $\omega_X$ is an \emph{extremal symplectic connection} if $$p(\Lap_\calV \Lambda_{\omega_B} \mu^* F_{\calH} + \Lambda_{\omega_B} \rho_{\calH}) \in \mfh_V,$$ where we view $\mfh_V$ as a space of holomorphy potentials through the relatively K\"ahler metric $\omega_X$.
\end{definition}

Since the projection operator $p$ is zero on pullbacks of functions from $B$, there is some redundancy in the existence of an optimal symplectic connection. Namely we note that combining this observation with \cref{thm:minimalcoupling} if $\omega_X$ is an optimal symplectic connection, so is $\omega_X + \pi^* \beta$ for any closed $(1,1)$-form $\beta$. In fact this redundancy and the action of the relative automorphism group of the fibration are the only degrees of freedom in choosing an optimal symplectic connection $\omega_X$. That is, optimal symplectic connections are unique in this sense:

\begin{theorem}[\cite{dervan2021uniqueness,hallam2020geodesics}]\label{thm:uniqueness}
	Suppose $\omega_X,\omega_X'$ are two cohomologous optimal symplectic connections on a holomorphic fibration $\pi: X\to (B,\omega_B)$ . Then there exists a holomorphic automorphism $g$ of the fibration $X\to B$, that is a biholomorphism $g: X\to X$ such that $\pi \circ g = \pi$, and a function $\varphi$ on $B$, such that
	$$\omega_X = g^* \omega_X' + \pi^* (i\deldelbar \varphi).$$
	Since the two forms are cohomologous, the automorphism $g$ can be taken to be in the identity component $\Aut_0(\pi)$ of the automorphism group $\Aut(\pi)$ of the fibration.
\end{theorem}

This uniqueness statement can be interpreted as stating that an optimal symplectic connection is a canonical relative K\"ahler metric on a holomorphic submersion.

\subsection{Hermite--Einstein principal bundles}

We now recall the theory of Hermite--Einstein connections on holomorphic principal bundles, as in \cite{anchouche2001einstein}, where a Hitchin--Kobayashi correspondence for such connections is proven. The basic theory of Hermitian connections on principal bundles is explained in \cite[Ch. IX \S10]{kobayashi1963foundations}.

\begin{definition}
	A \emph{Hermitian structure} on a holomorphic principal bundle $P\to (B,L)$ with reductive structure group $G$ is a reduction of structure group $\sigma:B \to P/K$ to a maximal compact subgroup $K\subset G$. The associated principal bundle with structure group $K$ is denoted $P_\sigma$.
\end{definition}

One should think of the reduction of structure group $\sigma$ as analogous to choosing a Hermitian metric on a holomorphic vector bundle. In particular when $G=\GL(r,\CC)$ and $K=\U(r)$ then the quotient $G/K$ consists of Hermitian inner products on $\CC^r$, and a choice of $\sigma$ is precisely a Hermitian metric on the associated vector bundle $E\to (B,L)$ given by the standard representation. Any holomorphic principal bundle with reductive structure group admits Hermitian structures, since the contractibility of $G/K$ implies that the quotient fibre bundle $P/K$ always admits sections.

\begin{definition}[{\cite[Ch. IX \S10]{kobayashi1963foundations}}]
	A principal bundle connection $A\in \Omega^1(P, \g)$ on a holomorphic principal bundle is:
	\begin{itemize}
		\item \emph{Complex} if $A$ is of type $(1,0)$.
		\item \emph{Unitary with respect to a Hermitian structure $\sigma$} if $A$ is the pushforward of a principal bundle connection $A_\sigma$ on $P_\sigma$ under the associated bundle construction $P=P_\sigma \times_K G$ for the reduction of structure group $\sigma$.
	\end{itemize}
\end{definition}

We note that analogously to the existence of Chern connections, one has the following.

\begin{lemma}[{\cite[Thm. IX.10.1]{kobayashi1963foundations}}]
	Given a Hermitian structure $\sigma$ on a holomorphic principal bundle $P\to (B,L)$ with reductive structure group, there is a unique complex connection $A$ unitary with respect to $\sigma$, called the \emph{Chern connection} for $\sigma$.
\end{lemma}

Due to this lemma we can either think in terms of Hermitian structures $\sigma$ or complex unitary connections $A$, the two perspectives being equivalent. 

Turning to Hermite--Einstein connections, we first note that there exists a trivial central subbundle $Z\subset \ad P$ whose fibre at each point is the centre $Z(\mfg)$ of the corresponding Lie algebra fibre of $\ad P$. This bundle is the pushdown of the trivial centre subbundle of $P\times \g$. Sections $\tau$ of $Z$ correspond to central sections of $P\times \g$ which are constant in the fibre direction (which is equivalent to equivariance under the $G$ action on $P$ for central elements of $\g$).

\begin{definition}
	A Hermitian structure $\sigma$ on a holomorphic principal bundle $P\to (B,\omega_B)$ with reductive fibre over a K\"ahler manifold is 
	 \emph{Hermite--Einstein} if, when $A$ is the associated complex unitary connection, there exists a holomorphic central section $\tau\in H^0(B,Z)\subset H^0(B,\ad P)$ of $\ad P$ such that
	$$\contr_{\omega_B} F_A = \tau$$
	where $F_A \in \Omega^2(B, \ad P)$ is the curvature of $A$. If instead $$D_A \contr_{\omega_B} F_A = 0$$ for $D_A$ the induced connection on $\ad P$, we say that $A$ is \emph{Yang--Mills}.
\end{definition}


We firstly note that if $Z(\mfg)$ is simple, then the Hermite--Einstein condition reduces to asking that $\contr_{\omega_B} F_A$ vanish. Indeed, since the bundle of central elements is trivial by hypothesis, the Hermite--Einstein condition forces the section $\tau$ to vanish.
 
 We also note that if $P$ admits a Yang--Mills connection, there is a maximal parabolic subgroup of $G$ with Levi subgroup $L=\oplus_i L_i$ for $L_i$ the corresponding Levi factors, such that each induced principal $L_i$-bundle with connection is Hermite--Einstein, and such that $P$ is actually isomorphic to the corresponding $L$-bundle. This is analogous in the vector bundle setting to having a direct sum of vector bundles with Hermite--Einstein connections of possibly differing slopes, and indeed is equivalent to this for the adjoint bundle $\ad P$ of $P$ as a simple consequence of \cite[Corollary 3.8]{anchouche2001einstein}.


We now briefly recall the Hitchin--Kobayashi correspondence for principal bundles. This relies on a notion of stability of principal bundles, which was first introduced by Ramanathan for the case of Riemann surfaces \cite{ramanathan1975stable}, and was extended to projective manifolds by Subramanian and Ramanathan and to arbitrary compact K\"ahler manifolds by Anchouche and Biswas. We will avoid the technicalities of precisely stating the notion of polystability of principal bundles since it will not be directly relevant to us, but in the case of $G=\GL(r,\CC)$ it corresponds exactly to the notion of polystability for the standard associated vector bundle.

\begin{theorem}[\cite{subramanian1988einstein,anchouche2001einstein}]
	A holomorphic principal bundle $G$-bundle on a compact K\"ahler manifold $(B,\omega_B)$ admits a Hermite--Einstein connection if and only if it is $[\omega_B]$-polystable.
\end{theorem}

\section{Isotrivial fibrations}\label{sec3}

We now turn to the study of isotrivial fibrations.

\begin{definition}
	A K\"ahler fibration is \emph{isotrivial} if it is a holomorphic fibre bundle. 
\end{definition}

Note that by the Fischer--Grauert theorem this is equivalent to asking that the fibres of the smooth K\"ahler fibration, which is a proper holomorphic submersion, are biholomorphic \cite{fischergrauert}.

We will be interested in smooth isotrivial relatively cscK fibrations $(X,\omega_X)\to (B,\omega_B)$. Let $H$ be relatively ample on $X$ and suppose $\omega_X\in c_1(H)$. Then for the model fibre $(Y,\omega_Y,H_Y)$ of $X$, where $\omega_Y$ is cscK, the automorphism group of $(Y,H_Y)$ is reductive \cite[\S 2.4, \S 8.1]{gauduchon2010calabi}. 

To summarise, let $G=\Aut_0(Y,H_Y)$ denote the the connected component of the identity of the group of holomorphic automorphisms of $Y$ which lift to $H_Y$, and let $K=\Isom_0(Y,\omega_Y,H_Y)$ denote the subset of $\Aut_0(Y,H_Y)$ of holomorphic isometries of $\omega_Y$. Then $K\subset G$ is a maximal compact subgroup. The Lie algebra $\h = \Lie(G)$ of complex holomorphy potentials of $Y$ can be identified with the holomorphic vector fields on $Y$ which vanish at least once. The Lie algebra $\k= \Lie(K) \subset \h$ is identified with the \emph{real} holomorphy potentials, and $\h = \k \oplus J\k$ where $J$ is the complex structure on $Y$. Integrating up to the Lie group, $G=K^\CC$, so $G$ is reductive.

In the case of fibrations the above description of the automorphism group has the following consequence.

\begin{lemma}\label{lemma:principalbundle}
	A smooth polarised isotrivial relatively cscK fibration $$\pi: (X,\omega_X,H)\to (B,\omega_B,L)$$ with model fibre $(Y,\omega_Y,H_Y)$ arises as the associated bundle to a reductive holomorphic principal $G=\Aut_0(Y,H_Y)$-bundle $P$ which admits a reduction of structure group $\sigma: B\to P/K$ to a principal $K=\Isom_0(Y,\omega_Y,H_Y)$-bundle $P_\sigma$ such that
	$$X = P\times_{G} Y = P_\sigma \times_{K} Y.$$
\end{lemma}
\begin{proof}
	Since $X\to B$ is a holomorphic fibre bundle, it admits a holomorphic system of local trivialisations, say $\{(U_\alpha, \varphi_\alpha)\}.$ Fix any $b_0\in B$ and any $\beta$ with $b_0\in U_\beta$. Define a model cscK metric $\omega_Y := \rest{{\varphi_\beta}_*}{b} \omega_X$ and polarisation $H_Y := \rest{{\varphi_\beta}_*}{b} H$. Let $G = \Aut_0(Y, H_Y)$ and $K= \Isom_0(Y, \omega_Y, H_Y)$. 
	
	By the uniqueness of cscK metrics up to automorphisms, for any local trivialisation $\psi$ on $V$ for $X$ and any $b\in V$, there exists a holomorphic automorphism $g_b$ of $Y$ taking $\rest{\psi_* \omega_X}{b}$ to $\omega_Y$. Since $\psi_* \omega_X$ and $\omega_Y$ are cohomologous (both lying in $c_1(H_Y)$), this may be taken to lie inside the reduced automorphism group of biholomorphisms which lift to the line bundle $H_Y$ for which $\omega_Y \in c_1(H_Y)$. Performed for the covering by the $U_\alpha$, this defines a system of local sections of a principal $G$-bundle over $B$, (since we have induced maps $U_{\alpha} \to G$ sending $b\in U_{\alpha}$ to the relevant $\psi\in\Aut_0(Y,H_Y)$). The cocycle condition for this system of sections follows from that of the trivialising functions $\varphi_\alpha$ for this cover. 
	
	Furthermore, if $b\in U_\alpha$ then under the identification of $X_b$ with $Y$ with respect to the biholomorphism $\varphi_\alpha$, the automorphism group $K$ for $Y$ can be identified with a conjugate of the isometry group $\Isom_0(X_b, \rest{\omega_X}{b}, H_b)$ in $\Aut_0(X_b, H_b)$. The cocycle condition for this system of local trivialisations guarantees that on overlaps of the $U_\alpha$ the isometry groups of $b\in U_{\alpha}\cap U_{\beta}$ are mapped to the same conjugate. This defines a smooth section $\sigma: B \to P/K$ specifying the desired reduction of structure group to $K$.
\end{proof}

\subsection{Induced optimal symplectic connections}

We suppose now that we have an isotrivial fibration arising as the associated holomorphic fibre bundle to some holomorphic principal bundle with reductive structure group, which reduces to a principal bundle for a maximal compact subgroup of the structure group. As we observed, every isotrivial relatively cscK fibration arises in this way, although in the following we make no assumption that the associated principal bundle has structure group the identity component of the automorphism group of the model fibre.

Explicitly, fix a smooth polarised variety $(Y,H_Y)$ with a constant scalar curvature K\"ahler metric $\omega_Y\in c_1(H_Y)$, and set $G = \Aut_0(Y,H_Y)$. By reductivity, $G$ is the complexification of $K=\Isom_0(Y,\omega_Y)$. The action of $K$ on $(Y,\omega_Y)$ is always Hamiltonian, where the Hamiltonian function of any induced vector field is given by the real holomorphy potential with respect to $\omega_Y$. Let us denote by $\mu: Y \to \k^*$ a corresponding moment map. Finally let us assume that $P$ admits a reduction of structure group $\sigma$ to $K$, and let $P_\sigma$ denote the principal $K$-bundle which induces $P$. Associated to this data is a holomorphic fibre bundle
$$\pi: X = P\times_G Y = P_\sigma \times_K Y \to B,$$
associated to $P$. Given such an associated bundle, there is an induced symplectic connection on $X$ given by the cscK metric on $Y$ and a complex unitary connection on $P$. This follows essentially from working with the smooth principal $K$-bundle $P_\sigma$ and applying a theorem of Weinstein (see for example \cite[Thm. 6.3.3]{mcduff2017introduction}). We reproduce the details here to emphasise the relationship between $P_\sigma$ and $P$ in our setting and sto how that the resulting form $\omega_X$ is a $(1,1)$-form on $X$.

\begin{proposition}\label{prop:inducedsymplecticconnection}
	Given the set up above, any choice of complex unitary connection $A$ on $P$ induces a symplectic connection $\omega_X\in c_1(H)$ on $X$.
\end{proposition}
\begin{proof}
	Let $v_\xi$ denote the induced vector field on $Y$ from some $\xi\in \k$ under the action of $K$ on $Y$. Then $v_\xi$ preserves the K\"ahler structure of $Y$. 
	
	Let $y\in Y, \hat y \in T_y Y, \eta \in \k$. Then we have the two identities
	$$\ip{d\mu(y) \hat y, \xi} = \omega_Y(v_\xi (y), \hat y),$$
	$$\ip{\mu(y), [\xi,\eta]} = \omega_Y(v_\xi(y),v_\eta(y)).$$
	The first follows from the definition of a moment map, and the second from the infinitesimal equivariance condition on the moment map.
	
	Let us abuse notation by writing $A=A_\sigma$ to denote the connection on the reduction of structure group $P_\sigma$ of $P$, and let $F\in \Omega^2(P_\sigma;\k)$ denote the curvature form of $A_\sigma$. Then for $p\in P_\sigma, \xi \in \k, v_1,v_2\in T_p P_\sigma$ we also have the standard expressions
	$$A_p(p\xi) = \xi,$$
	$$F_p(v_1,v_2) = (dA)_p(v_1,v_2) + [A_p(v_1),A_p(v_2)].$$
	Define a projection operator $\pi_A: TP_\sigma\times TY \to TY$ by $\pi_A(v,\hat y) = \hat y + v_{A_p(v)}(y).$ Let us define a two-form $\hat \omega_X\in \Omega^2(P_\sigma \times Y)$ by
	$$\hat \omega_X := \omega_Y - d\ip{\mu,A}.$$
	Then note that $\hat \omega_X$ is closed on $P_\sigma \times Y$.

	One may write 
	\begin{equation}\label{eqn:formequality}
		\hat \omega_X = \pi_A^* \omega_Y - \ip{\mu, F}.
	\end{equation}
	Indeed, using the identities above and the definition of $\pi_A$, we compute
	\begin{align*}
		&(\pi_A^* \omega_Y - \ip{\mu, F})_{(p,y)} ((v_1,\hat y_1), (v_2,\hat y_2))\\
		&= \omega_Y(\hat y_1 + v_{A_p(v_1)}(y), \hat y_2 + v_{A_p(v_2)} (y)) - \ip{\mu(y), F_p(v_1,v_2)}\\
		&= \omega_Y (\hat y_1, \hat y_2) + \omega_Y (v_{A_p(v_1)}(y), \hat y_2) - \omega_Y (v_{A_p(v_2)}(y), \hat y_1)\\
		&\qquad + \omega_Y (v_{A_p(v_1)}(y), v_{A_p(v_2)}(y))\\
		&\qquad - \ip{\mu(y), (dA)_p(v_1,v_2)} - \ip{\mu(y), [A_p(v_1),A_p(v_2)]}\\
		&=\omega_Y (\hat y_1, \hat y_2) + \ip{d\mu(y)\hat y_2, A_p(v_1)} - \ip{d\mu(y) \hat y_1, A_p(v_2)}\\
		&\qquad + \ip{\mu(y), [A_p(v_1),A_p(v_2)]}\\
		&\qquad - \ip{\mu(y), (dA)_p(v_1,v_2)} - \ip{\mu(y), [A_p(v_1),A_p(v_2)]}\\
		&= \omega_Y(\hat y_1, \hat y_2) - d\ip{\mu, A}_p ((v_1, \hat y_1), (v_2,\hat y_2)).
	\end{align*}

	From \eqref{eqn:formequality} it follows that the two-form $\hat \omega_X$ is $K$-invariant and horizontal for the quotient map $P_\sigma \times Y \to P_\sigma \times_K Y = X$. In particular a vector $(v,\hat y)$ is vertical with respect to this projection precisely if $v$ is a vertical tangent vector to $P$ and $\pi_A(v,\hat y) = 0$. From this it follows immediately that $\iota_{(v,\hat y)} \hat \omega_X = 0$. Since $\hat \omega_X$ is also closed this holds infinitesimally, and $\hat \omega_X$ is basic. 
	
	Additionally, the moment map condition for $\mu$ and $\omega_Y$ implies the $K$-equivariant closedness of $\omega - \mu$, and combined with the Bianchi identity this implies the equivariant closedness of $\tilde \omega_X$.
	
	Therefore the $K$-invariant and equivariantly closed two-form $\tilde \omega_X$ descends to a closed two-form $\omega_X$ on the quotient $X = P_\sigma \times_K Y$. The explicit expression 
	$$\omega_X((v_1,\hat y_1),(v_2, \hat y_2)) = \omega_Y(\hat y_1 + v_{A_p(v_1)}(y), \hat y_2 + v_{A_p(v_2)} (y)) - \ip{\mu(y), F_p(v_1,v_2)}$$
	shows that $\rest{\omega_X}{b} = \omega_Y$ since a vertical vector of $X$ takes the form $(0, \hat y)$, and  $(\omega_X)_{\calH} = \mu^* F_\calH$, since a horizontal vector is given by $(v,0)$ where $v$ is horizontal on $P$. Since we assumed that the initial connection $A$ was complex, the curvature $F$ has type $(1,1)$, and so $\omega_X$ is a closed $(1,1)$-form on $X$.
\end{proof}

In order to investigate when the symplectic connection $\omega_X$ induced on $X$ is optimal, we will use the following fact about compact K\"ahler manifolds which will simplify the optimal symplectic connection equation in the isotrivial setting.

\begin{lemma}[See for example {\cite[Lem. 28]{szekelyhidi2012blowing}}]\label{lemma:riccihamiltonian}
	If $h$ is a Hamiltonian function for a K\"ahler metric $\omega$ on a compact K\"ahler manifold, with Hamiltonian vector field $v$, then $\Lap h$ is formally the Hamiltonian function for the (not necessarily symplectic) two-form $\Ric \omega = \rho$ with the same vector field.
\end{lemma}
\begin{proof}
	\begin{align*}
		2\iota_v \rho &= \iota_v (dJd \log \det \omega)\\
		&= \calL_v (Jd \log \det \omega) - d\iota_v (Jd \log \det \omega)\\
		&= -d (\calL_{Jv} \log \det \omega)\\
		&= d\Lambda \calL_{Jv} \omega.
	\end{align*}
	Here we have used that $v$ preserves $J$ and $\omega$, where $\Lambda$ is the trace with respect to $\omega$. But $\calL_{Jv} \omega = -2i\deldelbar h$ so $\iota_v \rho = d\Lap h$. 
\end{proof}

The key argument which demonstrates how the optimal symplectic connection equation simplifies for isotrivial fibrations is the following.

\begin{proposition}\label{lemma:horizontalcomponentricciform}
	Given a complex unitary connection on a holomorphic principal bundle $P\to (B,\omega_B)$ with reductive fibre $G$, and an associated K\"ahler fibration $X=P\times_G (Y,\omega_Y)$ with cscK fibre, the relative Ricci form $\rho$ of the induced symplectic connection $\omega_X$ is related to the curvature of the connection $A$ by
	$$p(\contr_{\omega_B} \rho_\calH) = p(\Lap_\calV \contr_{\omega_B} \mu^* F_\calH).$$
\end{proposition}
\begin{proof}
	As above, suppose we have an action  $K\acts (Y,\omega_Y)$ of a real compact Lie group by holomorphic isometries on a K\"ahler manifold $Y$. By assumption this action admits an equivariant moment map $\mu: Y \to \k^*$. Let us define a comoment-type map 
	$$\nu^* : \k \to C_0^{\infty}(Y)$$
	by the composition
	$$\nu^* = \Lap \circ \mu^*.$$
	A restatement of \cref{lemma:riccihamiltonian} implies that the comoment map $\nu^*$ satisfies the standard moment map criterion with respect to the Ricci form $\Ric \omega_Y = \rho_Y$. That is,
	$$d\nu^*(\xi) = i_{v_\xi} \rho_Y.$$
	Additionally, since $K$ acts by isometries on $Y$, the Laplacian is $K$-equivariant as a morphism $C_0^{\infty} (Y) \to C_0^{\infty} (Y)$, and therefore the composition $\nu^* = \Lap \circ \mu^*$ is a $K$-equivariant map from $\k$ to $C_0^{\infty} (Y)$ with respect to the adjoint action of $K$ on $\k$. Differentiating this condition at the identity in $K$ gives the infinitesimal equivariance condition
	$$\nu^*([\xi, \eta]) = \rho_Y (v_\xi, v_\eta)$$
	for the comoment map $\nu^*$, which explicitly gives the interesting geometric formula
	$$\rho_Y(v_\xi, v_\eta) = \Lap(\omega_Y(v_\xi, v_\eta))$$
	for induced vector fields $v_\xi,v_\eta$. 
	
	Whilst the comoment map $\nu^*$ is not a genuine comoment map for a \emph{symplectic} form, it satisfies the same formal properties with respect to the $K$ action relative to $\rho_Y$ as $\mu^*$ does relative to $\omega_Y$. In particular the argument of \cref{prop:inducedsymplecticconnection} repeats without change for the differential form
	$$\tilde \tau_A := \rho_Y - d(\nu^* A) = \pi_A^* \rho_Y - \nu^* F$$
	on the product $P_\sigma \times Y$. Thus there exists a closed $(1,1)$-form $\tau_A$ on $X=P_\sigma \times_K Y$ with the property that $$(\tau_A)_\calV = (\rho)_\calV$$
	where $\rho$ is the relative Ricci form of $\omega_X$ itself. Additionally the explicit formula for $\tau_A$ reveals that
	$(\tau_A)_\calH = \nu^* F_\calH$.
	
	By \cite[Lem. 3.9]{dervan2019optimal}, if two closed $(1,1)$-forms agree when restricted to the vertical directions of a fibration, then their horizontal components are equal up to pullback from the base. In particular we have
	$$\rho_\calH + \pi^* \beta= \nu^* F_\calH$$
	for some two-form $\beta$ on $B$. 
	
	Let us now observe that after contracting with $\omega_B$ we have
	$$\contr_{\omega_B} \rho_\calH + \pi^* f= \contr_{\omega_B} \nu^* F_\calH = \Lap_\calV \contr_{\omega_B} \mu^* F_\calH$$
	for some function $f=\contr_{\omega_B} \beta$ on $B$. This second equality follows from the observation that $F_\calH$ on $X$ actually arises from a two-form defined on $B$, which is the very same $F_A\in \Omega^2(B, \ad P_\sigma)$ defining the curvature of the connection $A$ on $P_\sigma$ (after composing with the Lie algebra homomorphism from $\k$ to $\Ham(\calV)$). Since the two-form $\omega_B$ contracting $\mu^* F_\calH$ is also pulled back from the base, we can consider $\contr_{\omega_B} F_\calH$ as a section of the bundle of Hamiltonian vector fields on each fibre over $B$, and we have $\contr_{\omega_B} \mu^* F_\calH = \mu^* \contr_{\omega_B} F_\calH$ and similarly for $\nu^*$. Here we also use the identity
	$$\Lap_\calV \mu^* s = \nu^* s$$
	where $s: B\to \Ham(\calV)$ is a section of the relative Hamiltonian vector field bundle of the fibres of $X$ over $B$. This identity \emph{defines} $\nu^*$ for a general K\"ahler fibration, but in this case follows immediately from the descent of the comoment maps $\mu^*$ and $\nu^*$ to $X$ with respect to the diagonal action on $P\times Y$. 
	
	To conclude, we note that since the projection $p$ is invariant under the addition of a contraction of a form pulled back from the base, we have
	$$p(\Lap_\calV \contr_{\omega_B} \mu^* F_\calH) = p(\contr_{\omega_B} \nu^* F_\calH) = p(\contr_{\omega_B} \rho_\calH).$$
\end{proof}

In the following, recall that $\Aut_0(X,H)$ consists of the identity component of the group of automorphisms of $X$ lifting to $H$, while $\Aut_0(\pi) = \{g\in\Aut_0(X,H): \pi \circ g = g\}$.

\begin{theorem}\label{theorem:maintheorembody} Suppose $\Aut_0(\pi)$ is discrete. If a complex unitary connection $A$ on $P$ is Hermite--Einstein with respect to the Hermitian structure $\sigma$ defining the reduction of structure group to $K$, then the induced symplectic connection $\omega_X$ on $P$ is an optimal symplectic connection. 
\end{theorem}
\begin{proof}

We first note that, through the associated bundle construction, the automorphism group $\Aut_0(\pi)$ has Lie algebra satisfying $$\Lie \Aut_0(\pi) \cong H^0(B, \ad P),$$ where we use that $\Lie \Aut_0(\pi)$ consists of vertical holomorphic vector fields (holomorphic sections of $\ker d\pi$) vanishing on each fibre. By general theory of principal bundles, since $B$ is compact $$H^0(B, \ad P) \cong Z(\mfg),$$ so  $\Aut_0(\pi)$ being trivial is equivalent to asking that $Z(\mfg)$ vanish. In particular, the connection $A$ being Hermite--Einstein requires that $\contr_{\omega_B} F_A$ vanish.

	Let $(X,\omega_X)\to (B,\omega_B)$ be an isotrivial K\"ahler fibration with cscK fibres and base, with symplectic connection $\omega_X$ induced from a holomorphic principal bundle $P$. Then by \cref{lemma:horizontalcomponentricciform} the optimal symplectic connection equation for $\omega_X$ reduces to
	$$p(\Lap_\calV \contr_{\omega_B} \mu^* F_\calH) = 0.$$
	Suppose now that $\omega_X$ is an optimal symplectic connection, and that $\contr_{\omega_B} \mu^* F_\calH = h$ for some smooth function $h\in C^{\infty}(X)$. Then $h$ restricts to a mean-zero holomorphy potential on each fibre of $X$, because the isotrivial fibration $X$ arises from a holomorphic principal bundle and the curvature takes values in fibrewise holomorphic vector fields. Since $p$ is the orthogonal projection onto such relative holomorphy potentials, we have
	\begin{align*}
		0 &= \ip{h, p(\Lap_\calV h)}\\
		&= \int_X h p(\Lap_\calV h) \omega_X^m \wedge \omega_B^n\\
		&= \int_X h \Lap_\calV h \omega_X^m \wedge \omega_B^n\\
		&= \int_X |\nabla h|^2 \omega_X^m \wedge \omega_B^n.
	\end{align*}
	Thus the holomorphy potential $h$ is covariantly constant, and in fact zero as $\mu^*$ lands in the mean-zero holomorphy potentials. Thus we have $p(\Lap_\calV h) = 0$ if and only if $h$ is zero. In particular the optimal symplectic connection equation is equivalent to
	$$\contr_{\omega_B} \mu^* F_\calH = 0.$$
	
	Now if $\omega_X$ arose from a principal bundle connection $A$, then the correspondence between the Lie algebra of automorphisms and holomorphy potentials tells us that the above equation is equivalent to asking 
	$$\contr_{\omega_B} F_A =  0,$$
 so if $A$ is Hermite--Einstein then the induced symplectic connection $\omega_X$ is an optimal symplectic connection. 
\end{proof}

\begin{remark}\label{fano} The condition that  $\Aut_0(\pi)$ be  discrete can be removed if we assume the model fibre $(Y,-K_Y)$ is Fano. In this situation, the vertical Laplacian preserves fibrewise holomorphy potentials (i.e. each function lying in $C^{\infty}_E(X)$ is an eigenfunction of $\Delta_V$) \cite[Proposition 2.7]{dervan2021uniqueness}, so the condition $\contr_{\omega_B} F_A = \tau$ for $\tau \in Z(\mfg)$ implies that $\Delta_V \contr_{\omega_B} \mu^*F_{\calH} = \Delta_V\mu^*\tau \in \mfh_V$ (following the notation of \cref{osc-def}). Conversely, the same argument implies that if $\omega_X$ is an extremal symplectic connection arising from an associated bundle construction, that the connection $A$ must satisfy $\contr_{\omega_B} F_A = f(b)\tau$ for some function $f(b)\in C^{\infty}(B)$, and changing the Hermitian structure by a conformal factor (which corresponds to multiplying the transition functions defining the reduction of the structure group to $K$) produces a genuine Yang--Mills connection.

\end{remark}

As noted, the Hitchin--Kobayashi correspondence for principal bundles allows us to interpret the above theorem in terms of the algebraic geometry of $P$. In the following, take $G=\Aut_0(Y,H_Y)$ as before.

\begin{corollary}\label{corollary:stableprincipal}
	If a principal bundle $G$-bundle $P\to (B,L)$ is polystable over a cscK base with associated fibration  $\pi: (X,H)\to (B,L)$ satisfying $\Aut_0(\pi)$ is discrete, then the associated fibration admits an optimal symplectic connection.
\end{corollary}

By \cref{corollary:cscktotalspace} one can use \cref{theorem:maintheorembody} to generate new examples of cscK metrics. 

\begin{example}\label{example-application}
	Let $P\to (B,L)$ be a non-trivial, stable principal $\SL(2,\CC)$-bundle over a polarised variety $(B,L)$ of dimension at least two (every such principal bundle is trivial in dimension one, and the construction reduces to a product cscK metric in that case). Such a principal bundle could be constructed as the frame bundle of a non-trivial stable rank two holomorphic vector bundle over $(B,L)$. Assume $(B,L)$ admits a cscK metric and has discrete automorphism group. Let $(Y,-K_Y)$ denote the Mukai--Umemura threefold, which has automorphism group $\SL(2,\CC)$ and admits a K\"ahler--Einstein metric \cite[\S 5]{donaldson2008kahler}. Then the associated bundle $$X= P\times_{\SL(2,\CC)} Y$$
	admits an optimal symplectic connection by \cref{theorem:maintheorembody} and since the base has discrete automorphisms and $P$ is simple, the total space of the fibre bundle has discrete automorphisms and by \cref{corollary:cscktotalspace} admits a cscK metric. 
\end{example}

The construction demonstrated above is general, and produces a wide variety of new examples of cscK metrics on the total space of holomorphic fibre bundles.

\subsection{Induced Hermite--Einstein structures\label{section:OSCimpliesHE}}

Let $\pi: (X,\omega_X,H)\to (B,\omega_B,L)$ be an isotrivial relatively cscK fibration arising from a principal bundle $\varpi: P\to (B,L)$ as described by \cref{lemma:principalbundle}. In this section we will describe how to pass from the symplectic connection $\omega_X$ on $X$ to a principal bundle connection $A$ on $P$, and show that when $\omega_X$ is optimal, the induced principal bundle connection $A$ is Hermite--Einstein.

\subsubsection{Alternative description of $P$\label{section:alternativeprincipalbundle}}
First we proceed by giving an alternative invariant description of the principal bundle $P$ associated to the isotrivial fibration $X$. This description is inspired by the case of infinite-dimensional principal bundles of symplectomorphisms for a symplectic fibration \cite[Rmk. 6.4.11]{mcduff2017introduction}. 

Let $(Y,\omega_Y,H_Y)$ denote the model fibre of the isotrivial fibration. Then the fibre $P_b$ of $P$ over a point $b\in B$ is given by the set of all biholomorphisms $f: Y\to X_b$ isotopic to the identity which also lift to the linearisations $H_Y$ and $\rest{H}{b}$. This set is a $G$-torsor for the $G= \Aut_0(Y,H_Y)$ acting by precomposition, which defines the right $G$-action on $P$. 

The tangent space $T_f P\subset \Gamma(f^* TX)$ for some $f\in P_b$ is given by all vector fields $v\in \Gamma(\rest{TX}{X_b})$ such that $d(\pi \circ f)v : Y \to T_b B$ is constant and for which the vertical part of $v$ with respect to the symplectic connection $\omega_X$ is holomorphic. The vector fields which preserve the K\"ahler structure of the fibration further satisfy the compatibility condition that the one-form
$$\omega_X(v, df(-))\in \Omega^1(Y)$$
is closed, which implies the vector field preserves the symplectic form $\rest{\omega_X}{b}$.

The vertical vectors $V_f \subset T_f P$ consist of all vector fields of the form $v=df \circ u: Y\to \rest{TX}{X_b}$ for some $u\in \h(Y,H_Y)$ a holomorphic vector field on $Y$ which generates an automorphism lifting to $H_Y$. 

To define horizontal vectors, note that using the Ehresmann connection defined by $\omega_X$, any vector $v_0 \in T_b B$ admits a unique horizontal lift to a vector field $v_0^\# \in \Gamma(\rest{TX}{X_b})$. Define the horizontal vectors $H_f\subset T_f P$ as the vector fields of the form $v = v_0^\# \circ f$ for some $v_0^\#: X_b \to \rest{TX}{X_b}$ for $v_0\in T_b B$. 

To observe the splitting, note that any vector field $v\in T_f P$, when viewed as a vector field on $X_b$, admits a splitting with respect to $\omega_X$, and the condition that $d(\pi \circ f)v$ is constant is exactly the statement that the horizontal component of $v$ is of the form $v_0^\#$ for some $v_0\in T_b B$. 

Thus the connection $\omega_X$ on $X$ induces a connection on $P$. We note that the equivariance of $H\subset TP$ with respect to the action of $G = \Aut_0 (Y,H_Y)$ follows from the fact that $dR_g (v) = v_0^\# \circ (f\circ g)$ for $v=v_0^\# \circ f\in T_f P$ and $g\in G$, so $dR_g (H_f) \subset H_{f\cdot g}$.

Again we note that if we restricted to holomorphic isometries and vector fields preserving the K\"ahler structure then the same construction above would afford us a principal bundle connection on the principal $K$-bundle $P_\sigma$ which is the reduction of structure group for $P$ as in \cref{lemma:principalbundle}. This connection on $P_\sigma$ induces the connection on $P$ under the associated bundle construction, as can be seen easily by noting that the induced Ehresmann connection under the inclusion $K \hookrightarrow G$ simply views a horizontal vector $v=v_0^\# \circ f$ for $f:Y \to X_b$ a holomorphic isometry as a vector $v$ for $f: Y \to X_b$ a biholomorphism, forgetting the isometry. This clearly maps the horizontal subspaces for $P_\sigma$ into those for $P$.

\subsubsection{Curvature of the connection on $P$}

We have described how a symplectic connection $\omega_X$ on $X$ induces a principal bundle connection on $P$, which we denote by $A$. The curvature of $A$ is a two-form on $B$ with values in $\ad P$, a Lie algebra bundle with fibre $\h$, defined by
$$F_A(v_1,v_2) = [v_1^\#,v_2^\#]^{\text{vert}}\in \h(X_b,\rest{\omega_X}{b},\rest{H}{b}))$$
where $v_1,v_2\in T_b B$. 

The fact that $A$ is Hermite--Einstein follows by observing that since any two cscK metrics on the fibre $X_b$ may be transformed into each other by the action of an element $g\in \Aut_0(X_b, \rest{\omega_X}{b})$, up to a gauge transformation of the $G$-bundle $P$ we can assume that the symplectic connection $\omega_X$ on $X$ arises as the induced symplectic connection for the associated bundle construction $X=P\times_{G} Y$ exactly as in \cref{prop:inducedsymplecticconnection}. For such symplectic connections \cref{lemma:horizontalcomponentricciform} and \cref{theorem:maintheorembody} show that the conditions of being optimal is satisfied if \emph{and only if} $A$ is Hermite--Einstein. Thus we obtain:

\begin{theorem}
	An optimal symplectic connection $\omega_X$ on an isotrivial relatively cscK fibration $\pi: (X,H)\to (B,\omega_B,L)$, with $\Aut_0(\pi)$ discrete, induces a Hermite--Einstein connection on the associated principal bundle $P$ of relative automorphisms described in \cref{lemma:principalbundle} and \cref{section:alternativeprincipalbundle}.
\end{theorem}

We remark that from the construction of $P$ it is clear that a holomorphic automorphism of $P$ is the same data as a holomorphic fibre bundle automorphism of $X$, and the uniqueness of the Hermite--Einstein connection $A$ on $P$ up automorphism shows that the induced optimal symplectic connection on $X$ is unique up to automorphisms (and pullback of a two-form from $B$). This recovers a special case of the uniqueness result \cref{thm:uniqueness} of Dervan--Sektnan and Hallam.

\subsection{Product fibrations}

To conclude, we briefly discuss the case of product fibrations. The following holds for any product of holomorphic fibrations.

\begin{proposition}\label{prop:products}
	If $\omega_X$ and $\omega_{X'}$ are optimal symplectic connections on two fibrations $(X,H),(X',H')\to (B,L)$, then the product metric on the fibred product $Z=X\times_B X'$ is an optimal symplectic connection.
\end{proposition}
\begin{proof}
	This is simply a matter of verifying that the various terms appearing in the optimal symplectic connection equation split with respect to fibred products in the expected way. Let us first note that the curvature $F_\calH$ of the product metric $\omega=\omega_X+\omega_{X'}$ is the direct sum $F_\calH = F_X + F_{X'}$ where $F_X, F_{X'}$ denote the curvatures of $\omega_X, \omega_{X'}$ on $X$ and ${X'}$ respectively. One also observes that
	$$d\mu^* F_\calH = \iota_{F_\calH} \omega = \iota_{F_\calH} (\omega_X+ \omega_{X'}) = \iota_{F_X} \omega_X + \iota_{F_{X'}} \omega_{X'}$$
	from which we conclude that $\mu^* F_\calH = \mu_X^* F_X + \mu_{X'}^* F_{X'}$. Similarly it is straight forward to see $\Lap_\calV = \Lap_X + \Lap_{X'}$ where $\Lap_X$ and $\Lap_{X'}$ denote the vertical Laplacians on $X$ and $X'$ respectively. .
	
	If the dimension of the fibres of $X$ and ${X'}$ are $m_X, m_{X'}$ respectively, for the projection operator $p$ applied to a sum function $\varphi=\varphi_X + \varphi_{X'}$ where $\varphi_X, \varphi_{X'}$ depend only on the fibre coordinates of $X$ and of ${X'}$ respectively, we have
	\begin{align*}
		p(\rest{\varphi}{b}) &= \rest{\varphi}{b} - \frac{\int_{X_b\times {X'}_b} \rest{\varphi}{b} (\omega_X + \omega_{X'})^{m_X+m_{X'}}}{\int_{X_b\times {X'}_b} (\omega_X + \omega_{X'})^{m_X+m_{X'}}}\\
		&= \rest{\varphi}{b} - \frac{\int_{X_b\times {X'}_b} \rest{\varphi}{b} \omega_X^{m_X} \wedge \omega_{X'}^{m_{X'}}}{\int_{X_b\times {X'}_b} \omega_X^{m_X} \wedge \omega_{X'}^{m_{X'}}}\\
		&= \left( \rest{\varphi_X}{b} - \frac{\vol {X'}_b}{\vol {X'}_b}\frac{\int_{X_b} \rest{\varphi_X}{b} \omega_X^{m_X}}{\int_{X_b} \omega_X^{m_X}}\right) + \left( \rest{\varphi_{X'}}{b} - \frac{\vol X_b}{\vol X_b}\frac{\int_{{X'}_b} \rest{\varphi_{X'}}{b} \omega_{X'}^{m_{X'}}}{\int_{{X'}_b} \omega_{X'}^{m_{X'}}}\right)\\
		&=p_X(\rest{\varphi_X}{b}) + p_{X'}(\rest{\varphi_{X'}}{b}),
	\end{align*}
	where in the second to last step we have used Fubini's theorem.
	
	Finally, let us note that if $\rho_X$ and $\rho_{X'}$ denote the relative Ricci forms of $\omega_X$ and $\omega_{X'}$, then we have
	\begin{align*}
		\rho &= i \deldelbar \log \det (\omega_X + \omega_{X'})\\
		&= i \deldelbar \log \det \omega_X \det \omega_{X'}\\
		&= i \deldelbar \log \det \omega_X + i \deldelbar \log \det \omega_{X'}\\
		&= \rho_X + \rho_{X'}.
	\end{align*}
	Here we have used that the determinant of the product metric $\omega_X + \omega_{X'}$ is the product of the determinants, as can be seen by using the block matrix decomposition of the metric in local product coordinates, and note that any derivative in $\deldelbar$ in the $X$ fibre direction will vanish on $\log \det \omega_{X'}$, which depends only on the ${X'}$ fibre coordinates, and vice versa.
	
	In conclusion, since $\Lap_X \mu_{X'}^* F_{X'} = 0$ and $\Lap_{X'} \mu_X^* F_X=0$, we observe that
	\begin{multline*}
		p(\Lap_\calV \contr_{\omega_B} \mu^* F_\calH + \contr_{\omega_B} \rho_\calH)\\
		= p(\Lap_{\calV,X} \contr_{\omega_B} \mu_X^* F_X + \contr_{\omega_B} \rho_{H,X}) + p(\Lap_{\calV,{X'}} \contr_{\omega_B} \mu_{X'}^* F_{X'} + \contr_{\omega_B} \rho_{\calH,{X'}})
	\end{multline*}
	and thus we have the result.
\end{proof}

The above proposition combined with \cref{theorem:maintheorembody} produces a wealth of examples of optimal symplectic connections. Namely all fibred products of projectivisations of polystable vector bundles admit optimal symplectic connections, which indeed follows from the above proposition and previous work about the existence of optimal symplectic connections on projective bundles \cite{dervan2019optimal}.

This is somewhat curious from the perspective of vector bundles, as the direct sum of polystable vector bundles is not necessarily polystable (and therefore does not necessarily admit a Hermite--Einstein connection) unless the vector bundles have the same slope. Indeed the process of taking a direct sum and then projectivisation does not commute with taking projectivisation and then fibred product, in regards to the existence of optimal symplectic connections. On the other hand, a product of cscK manifolds is always cscK without regard for any topological matching criteria, so in this sense the study of optimal symplectic connections is closer to the study of cscK metrics. 

From the perspective of principal bundles, if $\fr(E)$ denotes the frame bundle of a vector bundle $E$, which is a principal $\GL(\rk E,\CC)$-bundle, then the fact that $\PP(E)\times_B \PP(F)$ admits an optimal symplectic connection when $E$ and $F$ are polystable corresponds to the fact that $\fr(E)\times_B \fr(F)$ is a polystable principal bundle. By the Hitchin--Kobayashi correspondence for principal bundles, this is known to be equivalent to the slope polystability of the slope zero vector bundle $$\ad (\fr(E)\times_B \fr(F)) = \ad \fr(E) \oplus \ad \fr(F) = \End(E) \oplus \End(F).$$ On the other hand if $E$ and $F$ have different slopes, so $\mu(E)<\mu(F)$ without loss of generality, then the principal bundle $\fr(E\oplus F)$ has adjoint bundle $\End(E\oplus F) = \End(E) \oplus \Hom(E,F)\oplus \End(F)$ which is not polystable, and so $\fr(E\oplus F)$ is not polystable either, which agrees with the fact that $\PP(E\oplus F)$ does not admit an optimal symplectic connection. 

It is known (see \cite{ramanan1984some}) that if $P$ is a polystable principal $G$-bundle and $\rho: G\to H$ is a representation which sends the connected component of the identity of the center, $Z_0(G)$, to the corresponding component $Z_0(H)$, then the associated principal $H$-bundle $P\times_\rho H$ is also polystable. Now the direct sum structure of $E\oplus F$ provides a reduction of structure group from $\fr(E\oplus F)$ to $\fr(E)\times_B \fr(F)$, however the associated homomorphism $\GL(\rk E,\CC) \times \GL(\rk F,\CC) \to \GL(\rk E + \rk F,\CC)$ does not map the centre into centre, as a central element
$$\begin{pmatrix}
	\lambda \id_E & 0\\
	0 & \mu \id_F
\end{pmatrix}$$
does not commute inside the larger group $\GL(\rk E + \rk F,\CC)$ unless $\lambda = \mu$. 

We also note more generally that by the uniqueness of optimal symplectic connections \cref{thm:uniqueness}, if a holomorphic fibration $(Z,H_Z)$ admits a fibred product decomposition into factors $(X,H_X)$ and $(X',H_X')$ admitting optimal symplectic connections, then any optimal symplectic connection on $Z$ in $c_1(H_Z) = c_1(H_X\boxtimes H_X')$ must be of fibred product form. If a \emph{reducible holomorphic fibration} is a fibration admitting such a fibred product decomposition, then one may restrict the study of optimal symplectic connections to irreducible holomorphic fibrations.

\bibliographystyle{alpha}

\bibliography{canonicalmetricsonholomorphicfibrebundles.bib}

\end{document}